\documentclass[reqno, 12pt]{amsart}

\usepackage{tikz-cd}

\usepackage{amsfonts, amsthm, amsmath, amssymb}
\usepackage{hyperref}
\hypersetup{colorlinks=false}
\usepackage[all]{xy}

\usepackage[margin=1in]{geometry}

\usepackage{helvet}

\RequirePackage{mathrsfs} \let\mathcal\mathscr
  
\usepackage{mathtools}
\mathtoolsset{showonlyrefs}

\usepackage{hyperref}
\usepackage{dsfont}
\usepackage{upgreek}
\usepackage{mathabx,yfonts}
\usepackage{enumerate}

\numberwithin{equation}{section}

\newtheorem{theorem}{Theorem}[section] 
\newtheorem{lemma}[theorem]{Lemma}
\newtheorem{proposition}[theorem]{Proposition}
\newtheorem{corollary}[theorem]{Corollary}

\theoremstyle{definition}
\newtheorem*{acknowledgements}{Acknowledgements}
\newtheorem{remark}[theorem]{Remark}
\newtheorem{definition}[theorem]{Definition}

\newtheorem*{notation}{Notation}

\renewcommand{\phi}{\varphi}

\renewcommand{\leq}{\leqslant}
\renewcommand{\le}{\leqslant}
\renewcommand{\geq}{\geqslant}

\renewcommand{\c}{\mathbf{c}}

\renewcommand{\b}{\mathbf{b}}

\renewcommand{\r}{\mathbf{r}}

\newcommand{\Zp}{\mathbb{Z}_{\text{prim}}}

\DeclareSymbolFont{bbold}{U}{bbold}{m}{n}
\DeclareSymbolFontAlphabet{\mathbbold}{bbold}

\renewcommand{\P}{\mathbb{P}}
\newcommand{\Q}{\mathbb{Q}}

\newcommand{\R}{\mathbb{R}}

\newcommand{\Z}{\mathbb{Z}}
\renewcommand{\l}{\left}

\renewcommand{\r}{\right}
\renewcommand{\b}{\mathbf}
\renewcommand{\c}{\mathcal}
\renewcommand{\epsilon}{\varepsilon}

\renewcommand{\leq}{\leqslant}
\renewcommand{\geq}{\geqslant}

\title{Finite saturation for unirational varieties}  

\author{Efthymios Sofos}
\address{
Max Planck Institute for Mathematics\\
Vivatsgasse 7, Bonn, 53111, Germany}
\email{sofos@mpim-bonn.mpg.de}

 \author{Yuchao Wang}
\address{Department of Mathematics, Shanghai University, Shanghai, 200444, China}
\email{yuchaowang@shu.edu.cn}

\date{\today}

\begin{document}

\begin{abstract}
We import ideas
from  
geometry 
to settle Sarnak's
saturation problem 
for a large class of algebraic varieties. 
\end{abstract}

\maketitle

\setcounter{tocdepth}{1}
\tableofcontents

\section{Introduction}
\label{introduction}
A topic of central importance in number theory regards
prime solutions of Diophantine equations.
A principal result 
in this area
is 
Vinogradov's three primes theorem~\cite{vinivedi},
proved via the Hardy--Littlewood method.
This approach has been extended
by Hua
to representations of integers by powers of primes,
see, for example,
the work of
Kawada and Wooley~\cite{MR1829558}
for further
developments
and references.
A recent prominent
work
is due to
Cook
and
Magyar~\cite{MR3279535}.
They obtained asymptotic
estimates
for the number of prime solutions
of
general
systems of Diophantine equations
under the assumption that the system has a large number of variables
compared to the degrees of the polynomials.
This method
has also been
utilised
in the subsequent work
of  Xiao--Yamagishi~\cite{xiya}
and Yamagishi~\cite{yama}.

An approach based on the circle method
demands that
the number of variables is rather large compared to the
degree. Thus when there are fewer variables,
investigations have focused on solutions with few prime factors,
a prototype result being that of Chen~\cite{MR0434997}
on the binary Goldbach problem, a result using the weighted sieve.
One can utilise
again the circle method
to cover cases where the number of
variables is moderately large,
see the
recent work of
Magyar--Titichetrakun~\cite{MR3656205},
Schindler--Sofos~\cite{scso}
and
Yamagishi~\cite{yamag}.
The natural barrier
for cancellation in exponential sums
prevents the
circle method
from working in a small number of variables;
for example,
in the case of hypersurfaces of degree 
exceeding $2$
it has never been used 
when the number of variables is less than twice the degree. 
In this realm one must necessarily
combine sieve techniques
with analytic methods other than the circle method.
There is only a small number of results
available in this range of variables.
Examples include
the work of
Marasingha~\cite{marasingha}
on homogeneous
ternary quadratic equations,
Wang's work
on the Fermat and the Cayley cubic surfaces
via universal torsors~\cite{MR3360328}
and the result of Liu and Sarnak~\cite{sarnak}
covering
non-homogeneous
ternary quadratic equations 
via the use of group actions
and the weighted sieve.
An important earlier work
is due to
Nevo
and 
Sarnak~\cite{MR2746350},
which covers
\textit{prime} entries of matrices
of fixed determinant.

Our aim in the present work is to
provide almost prime results
for varieties
in very few
variables
and with no underlying group structure.
We shall do so
by combining
sieve tools
with geometric arguments.
We first introduce the necessary notation.
Recall that the
function
$\Omega:\Z\to \mathbb{N} \cup \infty$
counting the
number of prime divisors with multiplicity
is defined via
\[
\Omega(m):=\sum_{p} \nu_p(m)
,\]
where $\nu_p$ is the standard $p$-adic
valuation,
the sum is over all primes $p$
and $\Omega(m)$ is infinite if and only if $m=0$.
We can extend $\Omega$ to
a function $\Omega_{\P^n(\Q)}$
defined on
$\P^n(\Q)$
in the obvious way,
namely
by finding for each $x\in \P^n(\Q)$
an element
$\b{x} \in \Zp^{n+1}$
with $x=[\b{x}]$
and letting
\[
\Omega_{\P^n(\Q)}
(x):=\Omega\l(
\prod_{i=0}^n
x_i\r)
.\]
Here
$\Zp^{n+1}$
stands for integer vectors $\b{x}\in \Z^{n+1}$
with $\gcd(x_0,\ldots,x_n)=1$;
it is easy to see that
$\Omega_{\P^n(\Q)}$ is well-defined.
\begin{definition}
[Saturation number]
\label{saturation}
Let $X\subset \P^n$ be
a variety defined over $\Q$.
The
saturation number
of $X$,
denoted by
$r(X)$,
is defined as the least
$r \in \mathbb{N} \cup \infty$
such that the set of
points
$x  \in X(\Q)$ with
$\Omega_{\P^n(\Q)}(x)\leq r$
forms a Zariski dense subset of $X$.
\end{definition}
This definition is essentially due to
Bourgain, Gamburd and Sarnak~\cite{bourg}.
Varieties $X$ with $X(\Q)$
being not Zariski dense will necessarily have
$r(X)=\infty$.
Let us note
that Zariski density is the geometric
analog of
the existence of infinitely many almost prime elements in a sequence of integers.
Indeed, if the Zariski density condition was relaxed
to existence of infinitely many points
$x  \in X(\Q)$ with
$\Omega_{\P^n(\Q)}(x)\leq r$ 
in the definition above
then the possible presence of (linear, for example)
subspaces defined over $\Q$
and contained in $X$ makes
the problem
more tractable, see~\cite{birch}. 
Recall that
a variety $X$
defined over $\Q$ is said to be
$\Q$-unirational
if there exists
a positive integer $m$
and
a dominant morphism
$\pi:\P^m \dashrightarrow X$
defined over $\Q$.
\medskip
\begin{theorem}
\label{unira}
Any smooth
projective
variety defined over $\Q$
has finite saturation
if it is
$\Q$-unirational.
\end{theorem}
\medskip
We shall prove this by using the dominant map to parametrise points and
then apply
 the weighted sieve to the forms associated to $\pi$.
Theorem~\ref{unira}
covers new
cases
where the number of variables is small.
We illustrate this with a few examples.
\medskip
\begin{corollary}
\label{cor:kolarary}
The following varieties have finite saturation:
\begin{itemize}
\item[({I})] Del Pezzo surfaces defined over $\Q$, of degree larger than $2$ and with at least one $\Q$-point;
\item[({II})] Smooth projective cubic hypersurfaces defined over $\Q$, of dimension larger than $1$ and with at least one $\Q$-point;
\item[({III})] Smooth projective hypersurfaces of odd degree $d\geq 5$, defined over $\Q$ and of dimension larger than $c=c(d)$;
\item[({IV})] Smooth projective quadric hypersurfaces defined over $\Q$, of positive dimension and with at least one $\Q$-point.
\end{itemize}
\end{corollary}
\medskip
The proof of the corollary
can be inferred immediately
from
Theorem~\ref{unira}
and the fact that the varieties in the statement are 
$\Q$-unirational.
Indeed,
unirationality follows from
the work of
Manin~\cite{maninakos}(part ({I})),
Koll{\'a}r~\cite{kollar} (part ({II}))
and
Brandes~\cite[Th.1.3]{brandes}
(part ({III})).
It is a standard result that 
all varieties in
part ({IV})
are $\Q$-rational,
in particular, they are $\Q$-unirational,
thus are covered by
Theorem~\ref{unira}.

There are additional results
on
unirationality
of hypersurfaces,
see,
for example,
the work of
Marchisio~\cite{MR1769987}
and
Conte, Marchisio, Murre~\cite{MR2527970},
respectively for
quartic and quintic hypersurfaces of small dimension.
It is worth mentioning that over 
algebraically closed fields, 
unirationality 
is known to 
hold for  
general 
hypersurfaces
of dimension
sufficiently large 
compared to the degree.
In the same setting,
Harris, Mazur and Pandharipande~\cite{MR1646558}
provided explicit criteria 
for unirationality of smooth hypersurfaces.

\subsection{Effective saturation}
One would wish to have an explicit bound on the saturation number
$r(X)$ in
terms of $X$.
In the generality of Theorem~\ref{unira}
this bound must necessarily depend
on the embedding of $X$.
Indeed,
as the example
$(m!)^2 x_0 x_1=x_2^2, (m \in \mathbb{N})$,
reveals,
the saturation $r(X)$ can be rather large, 
however, one could still try 
to bound it in terms of the coefficients. 
This is achieved in our next result.
\medskip
\begin{theorem}
\label{explicit saturation}
Assume that
$X\subset \P^n$
is smooth,
defined over $\Q$ and $\Q$-unirational,
so that
there exists $m\in \mathbb{N}$,
forms $f_i \in \Z[x_0,\ldots,x_m]$
and a
dominant rational morphism
$\pi:\P^m \dashrightarrow X$
given by
$\pi=(f_0:\ldots:f_n)$.
Letting $f:=\prod_{i=0}^n
f_i$,
the saturation  $r(X)$ is
at most
\[ 
10^5
 \deg^7(f)
\log (2\|f\| )
,\]
where  $\|f\|$ is
the
maximum of the 
absolute values of the coefficients of
$f$. 
\end{theorem}
It is possible that 
there exists a positive integer greater than $1$
that divides 
$f(\b{x})$
for every 
$\b{x}\in \Z^{m+1}$.
Therefore,
if one is allowed to make
linear
transformations to change the equations
defining $X$,
an improved saturation
can be obtained,
specifically, one depending solely on $\deg(f)$. 
\subsection{Improved saturation in the presence of a fibration with a $\Q$-section}
\label{s:usingsections}
Consider the
Fermat cubic surface
$X$,
given by
$\sum_{i=0}^3 x_i^3=0$.
It is well-known that it is 
$\Q$-birational to $\P^2$,
thus in particular it is $\Q$-unirational
and is therefore covered by 
Theorem~\ref{explicit saturation}.
One parametrisation was already given by Euler (see~\cite[\S 13.7]{MR2445243}),
however for our needs it will be more convenient to use the one 
given by
Elkies~\cite{elkies}.
More specifically,
he showed
that 
$X$
is dominated by
$\pi:\P^2 \dashrightarrow X$, where 
$\pi$
is given by 
$\pi=(f_0:\ldots:f_3)$
and 
\begin{align*}
&f_0=-(y_1+y_0)y_2^2 + (y_1^2+2y_0^2)y_2-y_1^3+y_0y_1^2-2y_0^2y_1-y_0^3,\\
&f_1=y_2^3 - (y_1+y_0)y_2^2 + (y_1^2+2y_0^2) y_2 + y_0y_1^2 - 2y_0^2y_1 + y_0^3,\\
&f_2=-y_2^3 + (y_1+y_0)y_2^2 - (y_1^2+2y_0^2) y_2 + 2y_0y_1^2 - y_0^2y_1 + 2y_0^3,\\
&f_3=(y_1-2y_0)y_2^2 + (y_0^2-y_1^2) y_2 + y_1^3 - y_0y_1^2 + 2y_0^2y_1 - 2y_0^3
.\end{align*} 
Thus Theorem~\ref{explicit saturation}
supplies us with the
estimate
$$r(X)\leq 
[10^5\cdot 12^7 \cdot (\log (2\cdot 2^4))] 
=(1.24\ldots) \times 10^{13}
.$$
The Elkies parametrisation 
provides $f_i$ with slightly
smaller coefficients
compared to the one given by Euler,
thus giving a 
smaller value for $\|f \|$
in Theorem~\ref{explicit saturation}
and hence 
a better value for $r(X)$.
When a subtler
geometric structure
is available then
one can hope to obtain an improved saturation number. 
Indeed, we shall see that
if a variety
can be covered
by
many
lower-dimensional
$\Q$-unirational
subvarieties
then we can
parametrise
the subvarieties
`uniformly'
and
apply sieve methods directly to them.
Potential
examples of such varieties are
those that are
equipped with a fibration that has a 
section over the base field. 
The approach of using sections to reduce the saturation 
is realised for a
class of varieties of dimension $2$
in our next result.
\begin{theorem}
\label{2 skew}
Every smooth cubic surface in
$\P^3_\Q$
that contains two rational skew lines
can be linearly
transformed over $\Q$
so that it has saturation at most $32$.
\end{theorem}
We have allowed for linear transformations
in order to eliminate the effect of small prime factors in the saturation.
Such surfaces have a conic bundle with a section over $\Q$;
indeed, the fibers
correspond to 
the residual conics in the pencil of planes through one of the skew lines
and
the other skew line ensures that each of these conics has a $\Q$-point,
see~\cite{MR3512880}.
It must be pointed out
that Theorem~\ref{2 skew}
does not cover the Fermat cubic surface.
This surface was proved to have saturation at most $20$
in a recent work of Wang~\cite{MR3360328}.

In the
`section-approach'
one can sometimes
replace
the sieve tools by
results originating in additive combinatorics.
This has the advantage of providing an
almost best possible
saturation number.
For example, we shall
use work of Green--Tao
and
Ziegler
to show
that
the smooth
cubic
surface given by
\begin{equation}
\label{ter}
X: \
(x_0-6x_1)x_2^2+36x_1x_2x_3+36(x_0+6x_1)x_3^2=x_0^2x_2+216x_1^2x_3
\end{equation}
has saturation $$r(X)\leq 10,$$
see
Remark~\ref{rem:greentaozieg}.
This should be compared with the recent work of
Tsang and Zhao~\cite{MR3641655},
a special corollary of which is that for every sufficiently large integer $N$
satisfying certain necessary congruence conditions,
the Lagrange equation
$$N=x_0^2+x_1^2+x_2^2+x_3^2$$
has a solution $\b{x}\in \mathbb{N}^4$ 
with $$\Omega(x_0x_1x_2x_3)\leq 16.$$

\subsection{Improved saturation in the presence of a fibration without a $\Q$-section}
\label{s:quasisections}
In the setting of Theorem~\ref{2 skew},
the variety is equipped with a conic bundle
that has a section over the ground field,
i.e. can be covered by curves
which \textit{always} have a $\Q$-point.
It is therefore desirable to ask whether a geometric approach can still
provide
a satisfying
saturation number in cases where
$X$ has a fibration without a section
over $\Q$.
We answer this affirmatively
in our next result. 
\begin{theorem}
\label{threefold}
The Fermat cubic threefold
\[
\sum_{i=0}^4x_i^3=0
\]
has saturation
at most $42$.
\end{theorem}
The method of proof consists of 
equipping the threefold 
with a section-less
conic bundle fibration over $\P^2$,
finding sufficiently many $\Q$-rational 
fibers,
and applying 
sieve machinery
to this thin subset of the fibers.
Progress on the arithmetic of cubic hypersurfaces of dimension
$3,4$ and $5$
has been 
sporadic,
see~\cite{MR3229043},~\cite{MR3361770} and the references therein
for progress in higher dimensions. 
Smooth 
cubic threefolds over $\Q$
that contain a line defined over $\Q$
have finite saturation by the second part of Corollary~\ref{cor:kolarary}.
To obtain a good bound on the saturation one can try to adopt the 
 approach
used in the proof of Theorem~\ref{threefold}
for every threefold in this family,
since they are always equipped with a conic bundle over $\P_\Q^2$.
We hope that the approach in the proof of 
Theorem~\ref{threefold}
will assist with future investigations 
regarding
Sarnak's saturation problem 
for varieties
$X$ defined over $\Q$
such that 
$X(\Q)$ is Zariski dense and $X$ is not $\Q$-unirational. 
For example,
Swinnerton-Dyer~\cite[\S 5]{MR0337951}
has given a way to describe curves of 
genus $0$
on
the K$3$ surface
\[x_0^4+x_1^4=x_2^4+x_3^4\]
and one might attempt 
to adopt
his method to
produce 
enough $\Q$-rational
curves of genus $0$
on quartic surfaces
with smaller Picard rank.
This would allow to parametrise
a large set of these curves
and then
follow the method in the proof of 
Theorem~\ref{threefold}.

\begin{notation}
\normalfont
For any functions
$f,g :[1,\infty) \to \mathbb{C},$
the equivalent notations
$f\!\left(x\r)\!=\!O_{\mathcal{S}}\left(g\!\left(x\r)\r),$
and
$f\!\left(x\r) \! \ll_{\mathcal{S}} \! g\!\left(x\r),$
will be used to denote the
existence of a positive constant
$\lambda,$ which depends at most
on the set of parameters $\mathcal{S}$
such that
for any
$x\geq 1$
we have
$\l|f\left(x\r)\r|\leq \lambda \l|g\left(x\r)\r|.$
As usual,
we denote
the M\"{o}bius function
by $\mu(n)$.
We denote the number of prime factors of $n$
without multiplicity
by 
$\nu(n)$.
We shall furthermore denote by $\text{Res}(f,g)$
the resultant of two integer  binary forms $f,g$,
and
by $\text{rad}\!\left(n\r)$
the square-free kernel of $n.$
Finally, the symbol $\|f\|$ is reserved for the
maximum absolute value of the coefficients of an integer polynomial $f$.
\end{notation}

\begin{acknowledgements}
\normalfont
We are indebted
to Tim Browning 
for suggesting the problems
that led to Theorems~\ref{unira},~\ref{explicit saturation}
and~\ref{threefold}.
Furthermore, 
we
wish to express our gratitude 
to
Jean-Louis Colliot-Th\'{e}l\`{e}ne,
Tim Browning and
Roger Heath-Brown
for several useful suggestions. 
This investigation was performed
while the 
first
author was supported by
the London Mathematical Society
via the
\texttt{150th Anniversary Postdoctoral Mobility Grant}
and the second author
was supported by NSFC (11601309).
\end{acknowledgements}

\section{Sieve preliminaries}
Before providing the proofs of our theorems let us record
certain auxiliary results.
We shall begin by stating a form of the weighted sieve.
Suppose that $\mathcal{A}$ is a subset of the integers
of cardinality $Y$
and that  
$\mathcal{P}$ is any set of primes 
which includes every sufficiently large prime. 
Our setting includes
a non-negative
multiplicative
arithmetic
function
$\omega$
that is 
supported on multiples of primes in
$\mathcal{P}$.
We define for integers $d$,
$\mathcal{A}_d:=\{a\in \mathcal{A}:d|a\}$
and we let $\mathcal{R}_d:=\#\mathcal{A}_d-\frac{\omega(d)}{d}Y$.
The following result appeared in
\cite[Th.11.1]{sievebook}
and in an earlier form in
\cite[Th. 0 \& Th.1]{applic}.
\begin{theorem}
[Diamond--Halberstam]
\label{appear}
Assume that there exist real constants
$\kappa>1$
and $A_1,A_2,A_3, A_4
\geq 2$
such that the following conditions are satisfied,
\[ a \in \mathcal{A}, p|a \Rightarrow p \in \mathcal{P},
\tag{A}\]
\[
0\leq \omega(p)<p \ \ \text{for all} \ p\in \mathcal{P},
\tag{B}\]
\[
\sum_{\substack{d < Y^\alpha (\log Y)^{-A_3} \\ p|d \Rightarrow p\in \mathcal{P}}}
\mu^2(d) 4^{\nu(d)} | \mathcal{R}_d| \le A_2\  Y/ \l(\log Y\r)^{\kappa+1}
\  \ \text{for some} \ 0<\alpha \leq 1,
\tag{C}\]
\[\text{there exists} \ \mu \in \R \ \text{such that} \
a \in \mathcal{A} \Rightarrow |a| \leq Y^{\alpha\mu},
\tag{D}\]
\[
\prod_{z_1 \le p < z} \left( 1 - \frac{\omega(p)}{p} \right)^{-1} \le
\left( \frac{\log z}{\log z_1}\right)^\kappa \left( 1 + \frac{A_1}{\log z_1}
  \right) \ \text{for} \ 2 \le z_1 < z.
\tag{E}\]
Then
there exists $\beta_\kappa\geq 2$
such that
whenever $u,v \in\R$ satisfy
$\alpha^{-1}<u<v,
\beta_\kappa<\alpha u$,
and
\[
\sum_{\substack{Y^{1/v}\leq p < Y^{1/u}\\
p\in \mathcal{P}}}
\#\mathcal{A}_{p^2}
\le A_4\  Y/ \l(\log Y\r)^{\kappa+1}
\tag{F},\]
then for any integer $r$
with
\begin{equation}
\label{r}
r>\alpha\mu u-1
+\frac{\kappa}{f_\kappa(\alpha v)}
\int_1^{v/u}
F_\kappa(\alpha v-s)\l(1-\frac{u}{v}s\r)
\frac{\mathrm{d}s}{s}
\end{equation}
we have
\[\#\{a \in \mathcal{A}: \Omega(a)\leq r\}
\gg
Y
\prod_{p<Y^{1/v}}\l(1-\frac{\omega(p)}{p}\r).
\]
The functions $f_\kappa,F_\kappa$
are defined as
the solutions to the delay-differential equations supplied in~\cite[Th.0]{applic}.
\end{theorem}
Condition $(F)$
is required
to ensure
that higher prime powers have little effect
in the calculations yielding the value of $r$
given in~\eqref{r}.
This condition is missing from
\cite[Th. 0 \& Th.1]{applic}
but appears in a slightly more general form in~\cite[(11.2)]{sievebook} as condition $\b{Q_0}$.
It is later stated in the same tract
\cite[pg. 140]{sievebook}
that for the purpose of proving
\cite[Th.11.1]{sievebook},
condition $\b{Q_0}$
is only needed to ensure condition $(F)$
as given in Theorem~\ref{appear}.

We shall furthermore use the bound
\begin{equation}
\label{beta}
\beta_\kappa\leq 3.75 \kappa,
\end{equation}
valid for all $\kappa>1$ and
proved in
Theorem $17.2$
and Proposition $17.3$
in~\cite{sievebook}.
Let us note that, as shown in~\cite[pg. 146]{sievebook},
when
$v=\beta_\kappa u$
then the quantity
\begin{equation}
\label{rr}
\mu-1
+(\mu-\kappa)(1-1/\beta_\kappa)+(\kappa+1)\log \beta_\kappa
\end{equation}
is an upper bound for the right
side of~\eqref{r}
and thus can be used in its place
in Theorem~\ref{appear}.

While the weighted sieve
is useful in situations where $\mathcal{A}$
is composed of values assumed by general integer forms,
in the special case that all
forms are linear
one can do significantly better
due to the groundbreaking work
of Green, Tao and Ziegler~\cite[Cor. 1.9]{t1},~\cite{t2} and~\cite{t3}.
We shall later use the following very special case of their work. 
\begin{theorem}[Green, Tao, Ziegler] \label{gtz}
Let $L_1,\ldots,L_5\in\Z[u,v]$ be
linear forms
which are
pairwise non-proportional
and
assume that for each prime $p$ there exists $(u_0,v_0)\in\Z^2$
such that
$p$ does not divide $L_i(u_0,v_0)$ for any $i=1,\ldots,5$.
Let $\mathcal{R}$ be a box of
$\mathbb{R}^2$
containing a point
$(u_0,v_0)\in\Z^2$
such that $L_i(u_0,v_0)>0$ for $i=1,\ldots,5$.
Then there exist infinitely
many
pairs
$(u,v)\in \mathcal{R}\cap\Z^2$ such that $L_i(u,v)$
are all prime.
\end{theorem}
Lemmas~\ref{lem:ABC}-\ref{lem:waltzfordebby} 
will be used 
in
the proof of Theorem~\ref{explicit saturation}
to make  the dependence of the saturation on the 
small prime factors completely explicit in terms of the underlying polynomials.
\begin{lemma}
\label{lem:ABC}
Assume that $a,b \in \R_{\geq 1}$
and that a positive integer $m$ satisfies 
\[
m\leq a
b^{\nu(m)}
.\]
Then we have 
$m\leq 
16
a^2
\exp(b^6)$.
\end{lemma}
\begin{proof}
If $b^{\nu(m)}\leq m^{1/2}$ holds,
then the assumption of the lemma ensures that $m\leq a^2$, which is sufficient for the proof.
If $b^{\nu(m)}\leq m^{1/2}$ fails, then 
$\log m \leq 2 (\log b)
\nu(m) $.
We are free to assume 
that $\log \log m>1$,
since otherwise we
have 
$m\leq e^e\leq 16$ and the lemma holds.
Then~\cite[Eq.(30)]{robin}
guarantees that 
$\nu(m)\leq 3 (\log m) (\log \log m)^{-1}$,
hence 
\[\log m
\leq 2 (\log b)
 3 (\log m) (\log \log m)^{-1}
.\]
Thus we obtain
$m\leq \exp(b^6)$,
which concludes the proof.
\end{proof}
For an integer polynomial of degree at least $2$
we denote its discriminant by $D_f$. 
\begin{lemma}
\label{lem:mahler}
Assume that $f\in \Z[x]$
is a non-constant polynomial
with non-zero 
discriminant
and of degree at least $2$.
Then the set 
\[
\{W\in \mathbb{N}: 
x\in \Z\Rightarrow 
W \text{ divides }
f(x)
\}
\]
is bounded.
Furthermore,
every
element in this set 
is at most 
\begin{equation}
\label{eq:johncoltrane}
16 D_{f} \exp(4096 (\deg(f))^6) 
.\end{equation}
\end{lemma}
\begin{proof}
In the case $\deg(f)=1$ it is easy to see that 
for every prime $p$ we have 
\[\nu_p(W)\leq \min\{\nu_p(f(0)),\nu_p(f'(0))\},\]
which
is sufficient
due to
$\|f\|=\max\{|f(0)|,|f'(0)|\}$.
Hence we can assume that $\deg(f)\geq 2$
for the rest of the proof.
First, write $f(x)=c_0 f_0(x)$,
where $c_0\in \Z\setminus \{0\}$ is the greatest common factor of the coefficients of $f$.
Every integer $W$ in the set of the lemma
must satisfy
\begin{equation}
\label{eq:mustsatf}
x\in \Z\Rightarrow
\frac{W}{\gcd(c_0,W)} \mid f_0(x)
.\end{equation}
Letting 
$W_0:=
W/\gcd(c_0,W)
$
we observe that
$D_{f_0}\neq 0$,
thus the work of 
Cameron~\cite[Eq.(43)]{MR1119199}
reveals that 
the number of solutions of 
\[
f_0(x)\equiv 0 \mod{W_0}
\]
is at most
\[
\prod_{p \mid W_0} \l(2 p^{\nu_p(D_{f_0})/2}+\deg(f)-2\r)
.\]
By~\eqref{eq:mustsatf}
the number of solutions is also equal to
$W_0$,
therefore we have
\[
W_0\leq \prod_{p \mid W_0} \l(2 p^{\nu_p(D_{f_0})/2}+2\deg(f)\r)
\leq 
D_{f_0}^{1/2}
(4 \deg(f))^{\nu(W_0)}
.\]
Combining the last inequality 
with Lemma~\ref{lem:ABC}
yields the following bound,
\[
W_0
\leq 16 D_{f_0}
\exp(4096 (\deg(f))^6)
.\]
Recall that the discriminant of a polynomial $g$
is an integer form in its coefficients that has degree 
$2(-1+\deg(g))$.
This shows that, in light of
$f=c_0 f_0$,
one has 
\[
D_{f}=c_0^{2(-1+\deg(f))}
D_{f_0}
,\]
therefore 
\[
W=
\gcd(c_0,W) W_0
\leq 
c_0 W_0 
\leq 16 
c_0^{3-2\deg(f)}
D_{f}
\exp(4096 (\deg(f))^6) 
.\]
Noting that 
$\deg(f)\geq 2$
guarantees that  
$c_0^{3-2\deg(f)}\leq 1$
concludes the proof of our lemma.
\end{proof}
We wish to obtain a version of 
Lemma~\ref{lem:mahler}
that is valid for polynomials in many variables.
To do so we shall assign small integer values to all variables except one and apply 
Lemma~\ref{lem:mahler}.
We first deduce a weaker, but more useful,
version of Lemma~\ref{lem:mahler}. 
\begin{lemma}
\label{lem:mmahl}
Assume that $f\in \Z[x]$
is a non-constant polynomial with non-zero discriminant.
Then every element in the set 
\[
\{W\in \mathbb{N}: 
x\in \Z\Rightarrow 
W \text{ divides }
f(x)
\}
\]
is bounded by
\[
\|f\|^{2\deg(f)}
 \exp(5000 \deg^6(f)) 
.\]
\end{lemma}
\begin{proof}
The case $\deg(f)=1$ is easy to handle
and is thus left to the reader.
In all other cases 
we use the first corollary in page $261$
of the work of 
Mahler~\cite{MR0166188},
which states that 
\[
|D_f|
\leq (\deg(f))^{\deg(f)}
L^{2\deg(f)-2}
,\]
where $L$
is the sum of the moduli of the coefficients of $f$.
The inequality
$$L\leq (1+\deg(f)) \|f\|$$
is obvious. Using $(k+1)^k \leq e k^k$, valid for every $k\in \mathbb{N}$,
we see that 
\[
|D_f|
\leq (\deg(f))^{3\deg(f)}
\|f\|^{2\deg(f)-2}
((1+\deg(f))/e)^{-2}
\]
and the bounds~\eqref{eq:johncoltrane},
$((1+\deg(f))/e)^{-2}\leq 2$
show that every element in the set 
in the statement of our
lemma is
bounded by 
\[
32
(\deg(f))^{3\deg(f)}
\|f\|^{2\deg(f)}
 \exp(4096 (\deg(f))^6) 
.\]
Using the inequality
$32
d^{3d}
\leq  
e^{4d^6}$,
valid for all
$d\in \mathbb{N}$,
concludes our proof.
\end{proof}

\begin{lemma}
\label{lem:waltzfordebby}
Let $F \in \Z[x_0,\ldots,x_m]$
be a polynomial of positive degree
that does not have repeated polynomial factors.
Then every element in the set 
\[
\{W\in \mathbb{N}: 
\b{x}
\in \Z^{m+1}
\Rightarrow 
W \text{ divides }
F(\b{x})
\}
\]
is bounded by
\[\|F\|^{2\deg(F)}
\exp(6000 \deg^6(F)) 
.\]
\end{lemma}
\begin{proof}
The case $m=0$
is contained in Lemma~\ref{lem:mmahl}.
We can therefore assume that $m\geq 1$ for the rest of the proof.
We may consider the discriminant of the polynomial
in the variable $x_0$
obtained
by fixing every other variable,
i.e. we bring into play the integer polynomial   
\[
H(x_1,\ldots,x_m):=
D_{F(x,x_1,\ldots,x_m)}
.\]
We note that $H$ is not identically zero,
since otherwise
there would exist a
non-zero integer $c$
and an
integer polynomial $h(x_1,\ldots,x_m)$
such that 
\[(c x-h(x_1,\ldots,x_m))^2 \text{ divides } F(x,x_1,\ldots,x_m)
\]
in the polynomial ring $\Q(x,x_1,\ldots,x_m)$.
This would contradict the assumption that $F$ does not have repeated 
polynomial factors.
Now, since $H(x_1,\ldots,x_m)$ is a non-zero polynomial
we have the trivial bound 
\[
\#\{\b{x} \in (\mathbb{N} \cap [1,B])^m: H(\b{x})=0\}
\leq 
(\deg (H))
B^{m-1}
, \ (B \in \mathbb{N}),
\]
obtained by fixing $m-1$ of the variables $x_i$.
If every 
$\b{x} \in (\mathbb{N} \cap [1,B])^m$ is a zero of $F$
then we infer
$
B\leq \deg(H)
$, 
therefore there exists $\b{y} \in \mathbb{N}^m$ 
with the properties
\[|\b{y}| \leq 1+\deg(H)
\text{ and }
D_{F(x,y_1,\ldots,y_m)}\neq 0
.\]
We fix this choice of $\b{y}$
and define 
$f(x)=F(x,y_1,\ldots,y_m)$.
We have 
$\deg(f)\leq \deg(F)$
and
\[
\|f\|\leq \|F\|
(1+\deg(H))^{\deg(F)}
.\]
Recalling that the discriminant of a polynomial of degree $d$
is a polynomial in the coefficients that has degree $2(d-1)$
we see that 
$\deg(H) \leq 2 (\deg(F)-1)$,
hence 
$
\|f\|\leq 
\|F\| (2 \deg(F))^{\deg(F)}
$.
As a final step,
we 
apply Lemma~\ref{lem:mmahl}
to $f$,
thus getting 
that every integer $W$ in our lemma
must be at most 
\[
\|f\|^{2\deg(f)}  \exp(5000 \deg^6(f)) 
\leq 
\|F\|^{2\deg(F)}  
(2 \deg(F))^{2\deg^2(F)}   
\exp(5000 \deg^6(F)) 
.\]
The inequality 
$(2d)^{2d^2} \leq \exp(1000 d^6)$,
valid for all
$d\in \mathbb{N}$,
concludes the proof.
\end{proof}
Finally, 
we shall need certain results on the number of zeros of affine varieties over finite rings.
\begin{lemma}
\label{lem:trivlegendre}
Let $F \in \Z[x_0,\ldots,x_m]$
be a primitive polynomial.
Then for all primes $p$ we have 
\[
\#\{\b{x} \in (\Z/p\Z)^{m+1}: F(\b{x})\equiv 0\mod{p}\}
\leq (\deg(F))
p^m
.\]
\end{lemma}
\begin{proof} 
A hypersurface $F=0$
can have at most $\deg(F)$ irreducible components,
thus one might try
to use the 
Lang--Weil estimate~\cite{langweil}.
However, 
their upper bound is of the form 
$\deg(F) p^m+O_F(p^{m-1/2})$,
and
for the purpose of proving Theorem~\ref{explicit saturation} 
we need 
explicit constants. 
We begin the proof by remarking that 
since the polynomial 
$F$
is primitive,
it contains a term of the form 
$c \prod_{i=0}^m x_i^{d_i}$,
where $c \in \Z$
is coprime with  $p$
and the $d_i \in \Z_{\geq 0}$ are such that 
$\sum d_i= \sum \deg(F_i)$
(here $c$ and $d_i$ depend on $p$).
We infer that there exists $i$ such that 
$d_i>0$, therefore 
we may view $F$
as a polynomial in the variable $x_i$
by fixing every other variable.
The fact that $c$ is coprime to $p$
ensures that this polynomial is not identically zero in $\mathbb{F}_p$,
thus fixing any of the variables 
$x_j, j\neq i$,
we obtain a non-zero element in $\mathbb{F}_p(x_i)$
that must therefore have at most $\deg(F)$ roots. 
\end{proof}
\begin{lemma}
\label{lem:bhargavakos}
Let $F \in \Z[x_0,\ldots,x_m]$
be a polynomial of positive degree
that does not have repeated polynomial factors.
Then for all primes $p$ we have 
\[
\#\{\b{x} \in (\Z/p^2\Z)^{m+1}: F(\b{x})\equiv 0\mod{p^2}\}
\ll_F
p^{2m}
.\]
\end{lemma}
\begin{proof}  
We begin by estimating the number of zeros of $F$ in $(\Z/p^2\Z)^{m+1}$
that are lifted from smooth zeros of $F$ in $\mathbb{F}_p$, we call this cardinality $X_p$.
By Hensel's lemma,
smooth zeros
of $F$
in $\mathbb{F}_p$
can lift to at most $\ll_{F} p^{m}$
zeros in $(\Z/p^\Z)^{m+1}$.
Thus, in light of Lemma~\ref{lem:trivlegendre}
we obtain
$
X_p \ll 
\#\{\b{x} \in \mathbb{F}_p^{m+1}:F=0\}
p^m
\ll_F 
p^{2m}
$.
Next, we estimate the number of zeros of $F$ in $(\Z/p^2\Z)^{m+1}$
that reduce to
singular zeros  of $F$ in $\mathbb{F}_p$, we call this cardinality $Y_p$.
Note that 
such zeros can certainly exist (e.g. when $F=F_1F_2$ then every non-zero
$\mathbb{F}_p$-point in $F_1=F_2=0$ is singular)
and these can lift in any number of ways to 
$\Z/p^2\Z$;
the trivial bound for such lifts is $\leq p^{m+1}$.
The number of singular zeros is
however
rather small:
the variety $Z$
given by
$F=\nabla F=0$
and defined over $\mathbb{F}_p$
is a subvariety of $\mathbb{A}^{m+1}(\mathbb{F}_p)$
that has codimension at least $2$ 
for all sufficiently large primes $p$.
This is because 
$F$
is separable over $\Q$, therefore it
reduces
to a 
separable polynomial in $\mathbb{F}_p$
for all sufficiently large primes $p$.
Now by the Lang--Weil estimate~\cite{langweil}
we see that $Z(\mathbb{F}_p)$
has at most 
$\ll_{F} p^{m-1}$
points,
therefore
$Y_p
\ll_{F}
p^{m-1}
\#\{\text{lifts}\}
\leq  
p^{2m}
$.
We obtain that the number of 
zeros of  
$F$ in $(\Z/p^2\Z)^{m+1}$
is $X_p+Y_p\ll_{F} p^{2m}$,
thus concluding our proof. 
\end{proof}

\section{The proof of Theorems~\ref{unira} and~\ref{explicit saturation}}
We commence by recording a
result regarding almost prime values of
completely general forms. The result
is a
straightforward
application of the weighted sieve.
\begin{theorem}
\label{weighted step}
Let $f
\in \Z[x_0,\ldots,x_m]$
be any form of positive degree.
Then there exists $r\in \mathbb{N}$
such that
for any non-empty
box $\mathcal{R}^{(0)}\subset \R^{m+1}$
we have 
\[
\#\Big\{\b{x} \in
\Z^{m+1}\cap \mathcal{R}^{(0)}B,
\Omega(
|f(\b{x})|)
\leq r\Big\}
\gg_{f,r,\mathcal{R}^{(0)}} 
B^{m+1}
(\log B)^{-\deg(f)},
\]
for all sufficiently large $B$.
Furthermore, one can take $r$ as the integer part of
\[
6
\deg^2(f) 
(\log 2\|f\| )
+
10^4
 \deg^7(f)
\]
\end{theorem}

\medskip
\medskip
Although we make no claim as to the given value of $r$
being best possible,
the example
$f(x_0,\ldots, x_m)=2^n(x_0\ldots x_m)^m$
illustrates that
$r$ must depend 
on $\|f\|$ and $\deg(f)$, 
unless the lower bound $B^{m+1+o(1)}$
is replaced by
$B^{m+o(1)}$.
This is because 
\[
\Omega(2^n(x_0\ldots x_m)^m)\geq n+ m(m+1)
\geq \frac{\log \|f\|}{\log 2}+\l(\deg f\r)^2
\]
holds 
for all integer vectors $|\b{x}|\leq B$,
except those at a set 
of size $O(B^{m})$.
For the
application towards Theorem~\ref{unira}
it is important to stress that we need a lower bound 
of the form $B^{m+1+o(1)}$ in Theorem~\ref{weighted step},
otherwise Zariski density might not be obtained.

Theorem~\ref{weighted step} is proved
immediately after the proof of the next proposition. 
\begin{proposition}
\label{weighted step 0}
Let $F_1,\ldots,F_g
\in \Z[x_0,\ldots,x_m]$
be non-zero,
primitive,
irreducible forms of positive degree
such that  
$F(\b{x}):=\prod_{i=1}^g F_i(\b{x})$
does not have repeated polynomial factors.
Then there exists $r_0\in \mathbb{N}$,
such that for
any non-empty
box
$\mathcal{R}^{(0)}\subset \R^{m+1}$
we have 
\[
\#\Big\{\b{x} \in
\Z^{m+1}\cap \mathcal{R}^{(0)}B,
\Omega(
|F(\b{x})|)
\leq r_0
\Big\}
\gg_{F_i,r_0,\mathcal{R}^{(0)}}
\frac{B^{m+1}}{(\log B)^{\deg(F)}},
\]
as $B\to +\infty$.
Furthermore, one can take $r_0$ as the integer part of
\[
4
\deg(F) 
(\log 2\|F\| )
+
10^4
 \deg^6(F)
.\]
\end{proposition}
\begin{proof} 
The main tool used in the proof is
Theorem~\ref{appear}.
Before using it we need to study the following function,
defined for $d\in \mathbb{N}$,
\[
\omega_0(d):=\frac{1}{d^m}
\#\Big\{
\b{x} \in (\Z/d\Z)^{m+1}:
\prod_{i=1}^g F_i
(\b{x})\equiv 0  \mod{d}
\Big\}.\]
Note that $\omega_0$ is multiplicative,
therefore 
Lemma~\ref{lem:trivlegendre}
shows that
for every $d\in \mathbb{N}$
we have 
\begin{equation}
\label{langweil}
\mu(d)^2
\omega_0(d)\leq 
\Big(
\sum_{i=1}^g 
\deg(F_i)
\Big)^{\nu(d)}
,\end{equation}
while Lemma~\ref{lem:bhargavakos}
ensures that 
for every prime $p$
we have 
\begin{equation}
\label{eq:mufavouritethings}
\omega_0(p^2)\ll_{F_i}
1
.\end{equation}

We now continue with the application of the weighted sieve
by introducing the related set-up.
Let $D$ be the maximum positive integer dividing  
$F(\b{x})$ for all $\b{x} \in \Z^{m+1}$.
By Lemma~\ref{lem:waltzfordebby}
we obtain 
\begin{equation}
\label{the bound}
D\leq
\|F\|^{2\deg(F)}
\exp(6000 \deg^6(F)) 
.\end{equation}  
Define
$
G(\b{x}):=
F(\b{x})/D
$
and observe that $G(\b{x})$
attains integer values for all $\b{x} \in \Z^{m+1}$.
The definition of $D$
shows that for every prime $p$
the integer $p^{1+\nu_p(D)}$
cannot divide $F(\b{x})$ for all $\b{x} \in \Z^{m+1}$. 
Hence for each $p|D$
there exists $\b{z}_p \in \Z^{m+1}$ with
$p^{1+\nu_p(D)}\nmid F(\b{z}_p)$
and
letting
\[
W:=\prod_{p|D}p^{1+\nu_p(D)},
\]
we get the existence of some $\b{z} \in \Z^{m+1}$ such that
\begin{equation}
\label{cd}
\b{x}\equiv \b{z} \mod{W}
\Rightarrow
\gcd(G(\b{x}),W)=1.
\end{equation}

We denote the set of
 prime divisors of $D$
by
$\mathcal{P}$
and
we
define the multiplicative
function
$\omega$
via
$\omega(p):=\mathbf{1}_{\c{P}}(p)
\omega_0(p)$.
Note that by~\eqref{langweil}
we have the following 
for all $d \in \mathbb{N}$,
\begin{equation}
\label{blue train}
\mu(d)^2
\omega(d)
\leq \l(\deg{F}\r)^{\nu(d)}
.\end{equation}
Next, we define
\[
\Psi:=\{\b{x} \in \Z^{m+1}:\b{x}\equiv \b{z} \mod{W}\}.
\]
We shall use
Theorem~\ref{appear}
to
sieve the multiset of integers given by
\[
\mathcal{A}
:=\{G(\b{x}): \b{x} \in
\Z^{m+1}
\!
\cap
B
\mathcal{R}^{(0)}
\!
\cap
\Psi
\}.\]
For all integers
$d$
coprime to $W$
define
$
\mathcal{A}_d:=
\{a \in \mathcal{A}:d|a\}
$.
For $d\in \mathbb{N}$ we have 
\[\#\mathcal{A}_d=
\sum_{\substack{\b{u}\mod{d}\\G(\b{u})\equiv 0 \mod{d}}}
\#\{
\b{x} \in \Z^{m+1}\cap B \mathcal{R}^{(0)},
\b{x}\equiv \b{z} \mod{W},
\b{x}\equiv \b{u} \mod{d}
\},
\]
thus a
simple counting argument
involving integer points in 
boxes 
gives
\begin{equation}
\label{cannonball2}
\#\mathcal{A}_d
=
\frac{\omega(d)}{d}
\frac{B^{m+1}\mathrm{vol}(\mathcal{R}^{(0)})}{W^{m+1}}
+O_{F_i,W,\mathcal{R}^{(0)}}
\big(
\omega(d)
d^m
+
\omega(d)
B^m
\big)
.\end{equation}
Invoking~\eqref{blue train}
and writing
\[Y:=\frac{B^{m+1}\mathrm{vol}(\mathcal{R}^{(0)})}{W^{m+1}}
,\]
allows us to 
infer
that for square-free $d$ that are coprime to $D$
one has 
\begin{equation}
\label{cannonball}
\mathcal{R}_d:=
\Big|
\#\mathcal{A}_d
-\frac{\omega(d)}{d}Y
\Big|
\ll
\l(\deg{F}\r)^{^{\nu(d)}}
\l(d^m+Y^{1-\frac{1}{m+1}}\r),
\end{equation}
where the implied constant depends at most on $F_i,\mathcal{R}^{(0)}$ and $W$.

Next,
we 
fix any values
$
\alpha \in (0,1/(m+1)), \
\kappa\geq \deg(F), \
\mu>\frac{\kappa}{\alpha(m+1)},
$
and
verify the conditions of
Theorem~\ref{appear}.
Condition
(A)
holds
owing to the observation that 
if $p \mid G(\b{x})$
for some $\b{x} \in \c{A}$
then $p\nmid D$
due to~\eqref{cd}.
Condition (B)
follows from
the fact that 
the definition of $D$ implies that 
if $p\nmid D$
then there exists $\b{x}\in \Z^{m+1}$
such that $p\nmid F(\b{x})$. 
Condition (C)
is satisfied as a consequence of~\eqref{cannonball}
while condition (D) is validated with any choice of $\mu>\frac{\kappa}{\alpha(m+1)}$.
Condition (E)
holds due to~\eqref{blue train}
and the estimate
\[
\prod_{z_1 \le p < z} \left( 1 - \frac{\kappa}{p} \right)^{-1} =
\l(\frac{\log z}{\log z_1}\r)^\kappa \l(1+O\!\l(\frac{1}{\log z_1}\r)\r)
.\]
Combining~\eqref{eq:mufavouritethings} with~\eqref{cannonball2},
one sees that
\[
\sum_{\substack{Y^{1/v}\leq p < Y^{1/u}\\
p\in \mathcal{P}}}
\!\!
\#\mathcal{A}_{p^2}
\ll
Y \big( 
\sum_{Y^{1/v}< p  }p^{-2}
\big)
+
Y^{1-\frac{1}{m+1}+\frac{1}{u}}
+Y^{\frac{2m+1}{u}}.
\]
The sum over $p$ is
$\ll Y^{-1/v}$, hence 
condition (F) is satisfied since
$2\leq \beta_\kappa<\alpha u$
combined with
$\alpha<1/(m+1)$
imply that
$u>2m+1$.
Therefore
Theorem~\ref{appear}
yields a finite value $r$
such that
\[
\#\{\b{x} \in \Z^{m+1}\cap \mathcal{R}^{(0)}B:
\Omega(G(\b{x}))\leq r\}
\gg_F
B^{m+1}(\log B)^{-\kappa},\]
as $B\to \infty$.
We can provide an explicit value for $r$,
namely $r$ can be be the least integer which is greater than  
 $
10^3 
\deg^6(F)
$
because 
it can be shown that $c_0$ is small enough so that 
\[
10^3 
\deg^6(F)
\geq 
2 \deg(F)
+
(1+\deg(F))
\log (3.75  \deg(F))  
\]
and, 
furthermore,
we can choose
$\kappa=\deg(F)$,
$\alpha$ sufficiently 
close to $1/(m+1)$
and
$\mu$
sufficiently close to $\kappa/(\alpha(m+1))$,
observe that 
\[
\mu-1
+(\mu-\kappa)(1-1/\beta_\kappa)+(\kappa+1)\log \beta_\kappa
\leq 
2 \deg(F)
+
(1+\deg(F))
\log (3.75 \deg(F))  
\]
and
allude to~\eqref{beta} and~\eqref{rr}.
Note that
\eqref{the bound}
reveals that
\[
\Omega(D)\leq \frac{\log D}{\log 2}\leq 
4
\deg(F) 
(\log \|F\| )
+
9000 \deg^6(F) 
,\]
which, when coupled with
the equality 
$\Omega(|F(\b{x})|)=\Omega(D)+\Omega(|G(\b{x})|)$,
proves our lemma.
\end{proof}
To prove Theorem~\ref{weighted step}
we factorise $f$ as
\[
f(\b{x})=
c\prod_{i=1}^g
F_i(\b{x})^{\nu_i} \ (\nu_i \in \mathbb{N}, c\in \Z- \{0\}),
\]
where each $F_i$ is a
primitive irreducible integer form of positive degree,
and then apply Proposition~\ref{weighted step 0} 
to
$F:=\prod_{i=1}^g F_i$.
Utilising
the inequalities
\[
\Omega\l(|f(\b{x})|\r)
\leq
\Omega\l(|c|\r)
+
\l(\deg{f}\r)
\Omega\l(|F(\b{x})|\r)
\]
and 
$\Omega(|c|)\leq (\log |c|)/(\log 2)\leq (\log \|f\|)/(\log 2)\leq 2
\log \|f\| $
yields the desired value of $r$.
\newline
\newline
We are now ready to prove
Theorems~\ref{unira} and~\ref{explicit saturation}.
\newline
\newline
It is sufficient to show that with the value of $r$ given in Theorem~\ref{explicit saturation} the
set of rational points $x \in X\!\!\l(\Q\r)$
satisfying $\Omega_{\P^n(\Q)}(x)\le r$,
is dense in the image of the $\Q$-points of 
$\P^m$
under the map $\pi$
and
with respect to the real analytic topology.
Let us therefore fix any $\epsilon>0$
and a rational point $\pi(y)$, where
$y=[\b{y}] \in \P_\Q^m$ for some integer vector $\b{y}=(y_i)_{i=0}^m$.
By continuity of the forms $f_j$
there exists $\delta>0$, which depends on
$\epsilon,f_i$ and $\b{y}$,
such that for any $\b{X} \in \R^{m+1}$
specified by
$|\b{X}-\b{y}|\leq \delta$,
one has
$|
f_j(X_i)
-f_j(y_i)
|<\epsilon
$
for all
$i$ and $j$.
Applying
Theorem~\ref{weighted step}  
to the form
$f:=\prod_{i=0}^n f_i$
and the box
$\mathcal{R}^{(0)}=\prod_{i=0}^m[y_i-\delta,y_i+\delta]$,
we are presented with a value of 
$r$ in the range  
\[
r\leq 
6
\deg^2(f) 
(\log 2\|f\| )
+
10^4
 \deg^7(f)
,\]
a rational point
$x=[\b{x}]$
satisfying 
$\Omega_{\P^n(\Q)}(x)
\leq r$
and 
whose image
is $\epsilon$-close to $\pi(y)$
due to 
\[
\Big|
\frac{f_j(x_i)}{B^{\deg{f_j}}}-f_j(y_i)
\Big|<\epsilon
.\] 
This
concludes the proof
of
Theorems~\ref{unira} and~\ref{explicit saturation}
owing to
\[
6
\deg^2(f) 
(\log 2\|f\| )
+
10^4
 \deg^7(f)
\leq 
10^5
 \deg^7(f)
\log (2\|f\| )
.\]

\section{The proof of Theorem~\ref{2 skew}}

Throughout this section $X$ will be a smooth cubic surface with two rational skew lines.
We begin by picking an appropriate model for $X$.
\begin{lemma}
[Minimal model]
\label{notdead}
There exist integers
$a_i,d_i,f_i \left(i=0,1\r)$
and
$b_j,e_j  \left(j=0,1,2\r)$
such that $X$ is given by $F=0$,
where
\begin{equation}
\label{model}
F=
a\!\left(x_0,x_1\r)\!x_2^2+
d\!\left(x_0,x_1\r)\!x_2x_3+
f\!\left(x_0,x_1\r)\!x_3^2+
b\!\left(x_0,x_1\r)\!x_2+
e\!\left(x_0,x_1\r)\!x_3,
\end{equation}
\[
\left( \begin{array}{ccc}
a\left(x_0,x_1\r) \\
d\left(x_0,x_1\r) \\
f\left(x_0,x_1\r)
\end{array} \right)
:=
\left( \begin{array}{ccc}
a_0 & a_1 \\
d_0 & d_1 \\
f_0 & f_1
\end{array} \right)
\left( \begin{array}{c}
x_0 \\ x_1
\end{array} \right)\]
and
\[
\left( \begin{array}{c}
b\left(x_0,x_1\r) \\
e\left(x_0,x_1\r)
\end{array} \right)
:=
\left( \begin{array}{ccc}
b_0 & b_1 & b_2 \\
e_0 & e_1 & e_2 \\
\end{array} \right)
\left( \begin{array}{c}
x_0^2 \\ x_0 x_1 \\ x_1^2
\end{array} \right).\]
Furthermore these integers
satisfy the following two properties.
\begin{enumerate}
\item The greatest common factor of the six integers
$a_i,d_i,f_i \left(i=0,1\r)$
is $1$
and this also holds for the six
integers
$b_j,e_j  \left(j=0,1,2\r).$
\item
The integers
$a_0b_0$ and $6$ are coprime
and each other of the aforementioned twelve integers is divisible
by
$6$.
\end{enumerate}
\end{lemma}
\begin{proof}
A linear change of variables over $\Q$
allows us to assume that
the
pair of skew lines
is given by
$x_0=x_1=0$
and
$x_2=x_3=0$.
Therefore $X$
is given by
$C_1+C_2=0$, where
$C_1$ is linear in $x_0, x_1$ and quadratic in $x_2, x_3$
and the opposite is true for $C_2$,
a statement which immediately shows
\eqref{model}.
It follows from the non-singularity of the surface that the 
polynomials $a$ and $b$ are not identically zero. 
Then under a suitable linear change of variables, if necessary, we may have $a_0b_0\neq0$.

Let $k_1$
and $k_2$
be the greatest common divisor of the
the six
integers
$b_j,e_j  \left(j=0,1,2\r)$
and
$a_i,d_i,f_i \left(i=0,1\r)$
respectively.
Then the linear change of variables
\[\left(x_0,x_1,x_2,x_3\r)
\mapsto
\left(k_2 x_0,k_2 x_1,k_1 x_2,k_1 x_3\r)
\]
shows that we can assume that $k_1=k_2=1,$
which proves $\left(1\r)$.
In order to prove claim
$\left(2\r)$, let us define the integers
\[\alpha=1+\max\{\nu_2(a_0)+\nu_2(b_0),2\nu_2(b_0)\} 
\text{ and }
\beta=1+\max\{\nu_2(a_0)+\nu_2(b_0),2\nu_2(a_0)\}.\]
Then the transformation
\[\left(x_0,x_1,x_2,x_3\r)
\mapsto
\big(
2^{\nu_2(a_0)} x_0,2^\beta x_1,2^{\nu_2(b_0)} x_2,2^{\alpha} x_3
\big)
\]
reveals that
property $\left(2\r)$ holds for the prime $2$.
An identical argument
for the prime $3$
concludes the proof of the lemma.  
\end{proof}
We turn our attention to the
conic bundle structure present in $X.$ Writing $F=x_0Q_0-x_1Q_1,$
where $Q_i$ are integral
quaternary quadratic forms,
we see that $X$ is equipped with a dominant morphism
$\pi:X \to \P_{\Q}^1 $
such that
\[ \pi\!\left(x\r)=\begin{cases}
   [x_0:x_1] & \text{if} \ \left(x_0,x_1\r)\neq 0,  \\
   [Q_1\!\left(x\r):Q_0\!\left(x\r)] & \text{if} \ \left(Q_1\!\left(x\r),Q_0\!\left(x\r)\r) \neq 0.
  \end{cases}
\]
Lemma~\ref{notdead} reveals that the fibres $\pi^{-1}\!\left(s:t\r)$
are the conics $Q_{s,t}=0$
given by
\[ Q_{s,t}(x,y,z):=
a\!\left(s,t\r)\!x^2+
d\!\left(s,t\r)\!xz+
f\!\left(s,t\r)\!z^2+
b\!\left(s,t\r)\!xy+
e\!\l(s,t\r)\!zy.
\]
Their discriminant
\[
\Delta\!\l(s,t\r)\!
=\!
\l(
ae^2 + f b^2- b d e
\r)\!
\l(s,t\r)
\]
is a quintic binary form and
is separable
owing to the non-singularity of $X$ (see
\cite[80, II.6.4, Proposition 1]{shafa}).
\begin{lemma}
The following polynomials are not identically zero,
\[a,b,e,f,d^2-4af.
\]
Furthermore the forms $a,d,f$ do not share a common polynomial
divisor and the same holds
for $b$ and $e$.
\end{lemma}
\begin{proof}
If the forms $a,d,f$ had a common divisor in $\Q[s,t]$ then the polynomials
\[p_1\!\l(x,y\r)=
a_0 x^2+
d_0 xy+
f_0 y^2
\
\text{and}
\
p_2\!\l(x,y\r)
=
a_1 x^2+
d_1 xy+
f_1 y^2,
\]
would share a common
zero $[\alpha:\beta] \in \P^{1}\!\l(\overline{\Q}\r)$.
We would then obtain a singularity on the cubic surface,
since
$
\nabla{F}(0,0,x,y)
=
\l(
p_1\l(x,y\r),
p_2\l(x,y\r),
0,0
\r)$.
If $a=0$ or $f=0$ then the partial derivatives of $F$ vanish
at the points $(0,0,1,0)$
and $(0,0,1,0)$ respectively.
These observations prove that $d^2\neq 4af$ since $\Q[s,t]$ is a unique factorisation ring.
The fact that the discriminant $\Delta\!\l(s,t\r)$
is separable   implies that
$b$ and $e$ are coprime forms.
In particular they can neither be identically zero, which finishes our proof.
\end{proof}
We deduce that
the
resultant
$W_0:=
\text{Res}\!\l(b,e\r)$
is a non-zero integer.
We furthermore get that
at least one
pair of the linear forms
$a,d,f$ has a non-zero resultant, say $W_1$.

Define for $(s,t) \in \Zp^2$
the forms
\[G_{s,t}\l(u,v\r):=b\l(s,t\r)u+e\l(s,t\r)v
\
\text{ and }
\
H_{s,t}\l(u,v\r):=a\l(s,t\r)u^2+d\l(s,t\r)uv+f\l(s,t\r)v^2,\]
where the polynomials
$a,\ldots,f$ were defined in Lemma~\ref{notdead}.

Each conic $Q_{s,t}=0$
contains the obvious point $(0,1,0)$,
a fact which can be used to provide
the following parametrisation of $Q_{s,t}(x,y,z)=0$,
\[
(x,y,z)=
(u G_{s,t}(u,v),
-H_{s,t}(u,v),
v G_{s,t}(u,v)).
\]

The weighted sieve
will allow us to
to
show that for appropriate values
of
$s,t$
both forms $G_{s,t}(u,v)$ and $H_{s,t}(u,v)$ attain almost prime values.
We begin by establishing certain facts required for the sieving process.
\begin{lemma}
\label{prime>3}
For each prime $p$
there exists $s_0,t_0\in
\l(\Z/p\Z\r)^\ast$
such that
both integers
$$
\gcd(b\!\l(s_0,t_0\r),e\!\l(s_0,t_0\r))
\ \text{ and } \ 
\gcd(a\!\l(s_0,t_0\r),d\!\l(s_0,t_0\r),f\!\l(s_0,t_0\r))$$
are coprime to $p$.
\end{lemma}
\begin{proof}
The second part 
of Lemma~\ref{notdead}
implies that one can take
$s_0\equiv t_0 \equiv 1 \mod p$
when $p=2$ or $3$.
Let $p$ be any prime larger than $3$.
The first part
of Lemma~\ref{notdead}
states that
$p$ cannot divide
all six
integers
$b_j,e_j  \l(j=0,1,2\r)$
so we may assume by symmetry that
$p
\nmid
\gcd(b_0,b_1,b_2).$
We may similarly
assume that
$p\nmid \gcd(a_0,a_1).$
We proceed to show the validity of
the bounds
\[\sharp\{s,t \in
\l(\Z/p\Z\r)^\ast\!\!, \ p|b(s,t)\}
\leq 2\l(p-1\r)
\text{ and } \ 
\sharp\{s,t
\in
\l(\Z/p\Z\r)^\ast
\!\!, \
p|a(s,t)\}
\leq p-1.
\]
The first bound is obvious in the case $p|(b_0,b_2).$
In the opposite case we may assume that $p\nmid b_0$
so
the number of solutions of
$b(s,t)\equiv 0 \mod{p}$
 is the same as
the number of solutions of
$
\l(2b_0s+b_1t\r)^2\equiv \l(b_1^2-4b_0b_2\r)t^2 \mod{p}$,
which is at most $2\l(p-1\r)$.
The second bound follows by the fact that $p\nmid \l(a_0,a_1\r).$
Combining the two bounds
implies that
\[\sharp\{
s,t \in
\l(\Z/p\Z\r)^\ast
\!\!, \
p|a\!\l(s,t\r)b\!\l(s,t\r) \}
\leq 3\l(p-1\r),
\]
which is strictly smaller than
$\l(p-1\r)^2$ for $p>3$
and that concludes the proof of the lemma.
\end{proof}

\begin{lemma}
\label{gauss}
There exists a square-free
$W$
and
integers
$s_0,t_0$
with
$(s_0t_0,W)=1$
such that
 whenever
the
coprime
integers $s$ and $t$
satisfy
$
(s,t)\equiv (s_0,t_0)\mod{W}
$,
then both integer forms
$G_{s,t},
H_{s,t}
$
are primitive.
\end{lemma}
\begin{proof}
We define $W$ as the product of the primes appearing in $W_0$ and $W_1.$
Combining
the Chinese remainder theorem
and Lemma~\ref{prime>3} for the primes dividing $W$
provides
congruence classes $\mod{W}$, say $s_0$ and $t_0$.
Assume that
the coprime
integers $s$ and $t$ satisfy
\[(s,t)\equiv (s_0,t_0)\mod{W}\]
and suppose for a contradiction that $G_{s,t}$ is not primitive
so that there exists a prime $p$ that divides both
$b(s,t)$
 and $e(s,t)$.
Elimination theory
applied to
$b(x,1),e(x,1)$
and
$b(1,x),e(1,x)$
shows that there are integer binary forms
$g_i(s,t)\ \l(i=1,\ldots,4\r)$
and non-negative
integers $\mu_j \ \l(j=1, 2\r)$ such that
\[b(s,t) g_1(s,t)+e(s,t)g_2(s,t)=W_0t^{\mu_1}
\
\text{ and } 
\
b(s,t) g_3(s,t)+e(s,t)g_4(s,t)=W_0 s^{\mu_2}.\]
The coprimality of $s$ and $t$ implies that
$p$ divides $W_0$ and this contradicts the choice of $s_0,t_0$ by Lemma~\ref{prime>3}.
A similar argument works proves that $H_{s,t}$ is
also primitive
and thus the proof is complete.
\end{proof}
Define the following binary integer forms,
\begin{align*}
\Phi_1\l(u,v\r):&=u,\\
\Phi_2\l(u,v\r):&=v,\\
\Phi_3\l(u,v\r):&=bu+ev,\\
\Phi_4\l(u,v\r):&=
au^2
+duv
+fv^2
\end{align*}
and let
\[\Phi:=
\Phi_1
\Phi_2
\Phi_3^2
\Phi_4^2.
\]
Our aim henceforth is
to show that the
form $\Phi\!\l(u,v\r)$
attains
almost prime values infinitely often.
To this aim it is natural to
assume the following local conditions,
\begin{alignat}{2}
&6|\l(e,d,f\r), \ &
\
&\gcd(6,ab)=1, \label{acdc1}\\
&\gcd(b,e)=1,\   &
\
&\gcd(a,d,f)=1, \label{acdc2}\\
 &D \  \text{is a non-zero integer,}\label{acdc3}
\end{alignat}
where
\[D:=
\text{rad}
\Big(
2\!\cdot\!
3\!\cdot\!
5\!\cdot\!
ab
ef
\!
\l(d^2-4af\r)
\l(ae^2 + f b^2- b d e\r)
\Big).
\]
There are two cases to consider according to the splitting behavior of the quadratic form $\Phi_4$
over the rationals. Although both cases can be treated via the weighted sieve
we shall use
Theorem~\ref{gtz}
for the former case. The reason is that this approach provides special examples of smooth cubic surfaces (see~\eqref{ter})
for which the saturation number is particularly small.
As an application
we deduce the following lemma.
\begin{lemma}
\label{666}
Assume that the integers
$a,b,d,e,f$
satisfy
conditions
\eqref{acdc1}-\eqref{acdc3}
and that $d^2-4af$ is a
non-zero
integer square.
Let $\mathcal{R}$ be any
non-empty
box in $\R^2.$
Then
\[
\lim_{B\to+\infty}
\sharp\{
\l(u,v\r) \in
\Z^2
\!
\cap
\!
\mathcal{R}
B,
\
\Omega(\Phi(u,v))\leq 8
\}
=
+\infty
.\]
\end{lemma}
\begin{proof}
Assume that
$d^2-4af=\delta^2$ for some natural $\delta$.
This implies that there are integers $a_1,a_2$ with $a=a_1a_2$
that divide the integers
$\l(d-\delta\r)/2,
\l(d+\delta\r)/2$
respectively.
The identity
\[\Phi_4(u,v)=\l(a_2u+\frac{(d-\delta)}{2a_1}v\r)
\l(a_1u+\frac{(d+\delta)}{2a_2}v\r)\]
shows that $\Phi_4$ splits into two integer linear binary forms, say $L_4, L_5$.
Note that $\Phi_4$
has no fixed prime divisor owing to~\eqref{acdc1} and~\eqref{acdc2}
and thus this property also holds for both linear forms.

We apply Theorem~\ref{gtz}
with
$
L_i=\Phi_i$
for
$
i=1,2,3.
$
The condition~\eqref{acdc3}
guarantees that the linear forms are pairwise non-proportional.
Let us verify the remaining condition of
Theorem~\ref{gtz}
by contradiction. Assume that for every
prime $p$
and all integers $u,v$
there exists an index $1\leq i \leq 5$
such that $p|L_i(u,v).$
We may show that we only need to restrict attention to
the primes $p=2,3$ since if $p\geq 5$ then
\[
\sharp\{u,v \in \l(\Z/p\Z\r)^{*},p \nmid \l(L_3L_4L_5\r)\l(u,v\r)\}\leq 3\l(p-1\r),
\]
which is strictly smaller than $(p-1)^2.$
Regarding the primes $p=2,3$
we note that conditions~\eqref{acdc1}
and
\eqref{acdc2}
ensure that we get a contradiction.

Theorem~\ref{gtz} provides
a choice of
signs
$s_i \in \{1,-1\}, i=1,\ldots,5,$
such that the $5$ linear form $s_i L_i$ attain prime values infinitely often.
Recalling the extended definition of $\Omega$
given in the introduction proves the statement of our lemma.
\end{proof}

We next
apply a level of distribution result
\cite[Th.3]{marasingha}
to prove the corresponding result in the remaining case.
\begin{lemma}
\label{weighted}
Assume that the integers
$a,b,d,e,f$
satisfy
conditions
\eqref{acdc1}-\eqref{acdc3}, that $d^2-4af$ is not an integer square
and let
$\mathcal{R}$ be a box of $\R^2.$
Then
\[\sharp\{
\l(u,v\r) \in
\Z^2
\!
\cap
\!
\mathcal{R}
B
,
\
\Omega(\Phi(u,v))\leq 30
\}
\gg \frac{B^2}{\l(\log B\r)^4},
\]
for $B \geq 2$.
The implied constant in the lower bound
depends at most on the coefficients of $\Phi$ and
$\mathcal{R}$.
\end{lemma}
\begin{proof}
Arguing as in
Lemma~\ref{prime>3}
for every prime $p|D$
and using the Chinese remainder theorem,
we get that
there are
integers $u_0,v_0 \mod{D}$
such that
\begin{equation}
\label{coprim}
\gcd\!\l(
\Phi\!\l(u_0,v_0\r)
,D\r)=1.
\end{equation}
Define the integer binary form
$
\Phi^\prime
:=
\Phi_{1}
\Phi_2
\Phi_3
\Phi_4
$
and
the subset of $\Z^2$,
\[\Psi:=\{\l(u,v\r)
\in
\Z^2:
\l(u,v\r)
\equiv
\l(u_0,v_0\r)
\mod{D}
\}.
\]
We write
$
\mathcal{A}':=\{\Phi'\l(u,v\r): \l(u,v\r) \in \mathcal{R}
B\!
\cap
\Psi
\}
$
and we let for
every non-zero integer $d$,
\[
\mathcal{A}'_d:=\{
a' \in \mathcal{A}'
:d|a'
\}
.\]
Let for any
$\b{c}=\l(c_1,\ldots,c_4\r) \in \mathbb{N}^4$,
$\rho\!\l(\b{c}\r)$ be the
the multiplicative
function given by
\[\rho\!\l(\b{c}\r)
:=
\sharp\{
\l(u,v\r)
\!
\mod{c_1\ldots c_4}:
c_i|
\Phi_i\l(u,v\r)
,
i=1,\ldots,4
\}.
\]
Define the following sets of primes,
\[\mathfrak{P}:=\{p \ \text{prime} :p \nmid D\}
\
\text{and}
\
\overline{\mathfrak{P}}:=\{p \ \text{prime}  :p |D\}
\]
and define the multiplicative function $\varpi\!\l(n\r)$
to be
supported on  square-free integers
and
\[\varpi\l(p\r):=
\mathds{1}_{\mathfrak{P}}(p)
\
p
\!
\sum_{\substack{c_i|p
\\
i=1,\ldots,4
\\
p|c_1\ldots c_4}}
\mu\!\l(p\r)
\mu\!\l(c_1\r)
\ldots
\mu\!\l(c_4\r)
\frac{\rho\!\l(c_1,\ldots,c_4\r)}{\l(c_1\ldots c_4\r)^2},
\]
where
$\mathds{1}_{\mathfrak{P}}$
is the characteristic function of
the primes in $\mathfrak{P}.$
Let
\[
Y:=B^2\frac{\text{vol}\!\l(\mathcal{R}\r)}{D^2}\]
and define
for $d\in \mathbb{N}$
the quantity
\[
R_{d}:=
\sharp
\mathcal{A}'_d
-\frac{\varpi\!\l(d\r)}{d}
Y
.\]

We need to verify the conditions
$\l(\text{A}\r)$-$\l(\text{E}\r)$
appearing
in
\cite[Th.3]{marasingha}
as well as condition $Q_0$ appearing in
\cite[Eq.(11.2)]{sievebook},
which does not seem to appear in~\cite{marasingha}.

Condition $(Q_0)$: Here we have to verify that
the estimate
\[
\sum_{z
< p
\leq y}
\sharp
\mathcal{A}'_{p^2}
\ll
\frac{Y \log Y}{z}
+y
\]
holds
for all $2\leq z <y$
with an implied constant that is independent of $z,y$ and $Y$.
The equality
\[
\mathcal{A}'_{p^2}
=
\Big\{
(u,v) \in \mathcal{R}B:
(u,v)\equiv (u_0,v_0) \! \mod{D},
p^2|
\Phi_1(u,v)
\ldots
\Phi_4(u,v)
\Big\}
\]
shows by~\eqref{coprim}
that
$\mathcal{A}'_{p^2}$
is empty if $p|D$.
Note
that since $\gcd(D,\Phi(u_0,v_0))=1$
and
$(u,v)\equiv (u_0,v_0) \! \mod{D}$
we get that
$\gcd(D,\Phi'(u,v))=1$
and hence no pair of integers $(u,v)$ with $\Phi'(u,v)=0$
is
counted by
$\mathcal{A}'_{p^2}$.
We therefore obtain that
$\sum_{z
< p
\leq y}
\sharp
\mathcal{A}'_{p^2}$
is
\begin{align*}
\ll
\sum_{i=1}^4
\sum_{
\substack{
z< p\leq y
\\
p\nmid D
}}
&\sharp \Big\{
(u,v) \in \mathcal{R}B, \Phi'(u,v) \neq 0,
\gcd(D,\Phi_i(u,v))=1,
p^2|\Phi_i(u,v)
\Big\}
\\
+
\sum_{\substack{1\leq i,j \leq 4
\\
i \neq j}}
&\sum_{
\substack{
z< p\leq y
\\
p\nmid D
}}
\sharp
\Big\{
(u,v) \in \mathcal{R}B,
\Phi'(u,v) \neq 0,
p|(\Phi_i(u,v),\Phi_j(u,v))
\Big\}
.\end{align*}
Recall the definition of $D$ given immediately after~\eqref{acdc3}
and note that
the primes dividing the resultants of each pair of forms $\Phi_i,\Phi_j$
divides $D$.
Hence,
since $p\nmid D$ in the previous sums,
we see that
\[
\sharp
\Big\{
(u,v) \in \mathcal{R}B,
p|(\Phi_i(u,v),\Phi_j(u,v))
\Big\}
\ll
\sharp
\Big\{
(u,v) \in \mathcal{R}B,
p|(u,v)
\Big\}
\ll
\frac{B^2}{p^2}+1.
\]
For $i=1$
we have 
$
\sharp\{
(u,v) \in \mathcal{R}B, \Phi'(u,v) \neq 0,
p^2|\Phi_i(u,v)
\}
\ll
\sharp \{
(u,v) \in \mathcal{R}B, uv\neq 0,
p^2|u
\}$,
which
is at most $\ll B^2p^{-2}$
and one obtains the same upper bound in the case $i=2$.
For $i=3$ we have the bound
$
\ll
\sharp\{
(u,v) \in \mathcal{R}B, (bu+ev)v\neq 0,
p^2|bu+ev\}
$
and the change of variables
$x_1:=bu+ev,x_2:=v$
reveals that
the quantity is bounded by
$\ll
\sharp\{
|x_1|,|x_2| \ll
B, x_1x_2\neq 0,
p^2|x_1
\}$,
which 
is easily seen to be 
at most
$\ll B^2p^{-2}$.
For $i=4$  we have
\[
\sharp \Big\{
(u,v) \in \mathcal{R}B, uv\Phi_4(u,v) \neq 0,
\gcd(D,\Phi_i(u,v))=1,
p^2|\Phi_i(u,v)
\Big\}
\ll
\sum_{\substack{1\leq
p^2n\ll B^2 \\ \gcd(pn,D)=1}}
r_{D_0}(p^2n),
\]
where the function $r_{D_0}(n)$
denotes the number of
representations
of an integer $n$ by
all binary quadratic forms
forming a set of representatives of discriminant
equal to \[D_0:=\text{disc}(\Phi_4)=d^2-4af.\]
Note that $D_0$ is a non-zero integer which is not a square
since $\Phi_4$ is irreducible and that $\gcd(n,D_0)=1$ holds
in the previous summation since all primes dividing $D_0$
also divide $D$ by its definition.
Using the fact that $r_{D_0}$ can be written as a Dirichlet convolution
of $1$ and a quadratic character we obtain the trivial bound
$r_{D_0}(n)\ll \tau(n)$,
where $\tau$ denotes the classical divisor
function.
Writing
$n=p^\nu n_0$
with $p\nmid n_0$
we
obtain
\[
\sum_{\substack{1\leq
p^2n\ll B^2 \\ \gcd(pn,D)=1}}
r_{D_0}(p^2n)
\ll
\sum_{\nu=0}^{\infty}
\tau(p^{2+\nu})
\sum_{\substack{
n_0
\ll B^2/p^{2+\nu}
\\
p\nmid n_0}
}
\tau(n_0)
\ll
\frac{B^2 \log B}{p^2}
\sum_{\nu=0}^{\infty}
\frac{\tau(p^{2+\nu})}{p^{\nu}}.
\]
The function
$
\sum_{\nu>0}
(3+\nu)
p^{-\nu}
=
(3-2/p)(1-1/p)^{-2} 
$
has modulus 
at most $8$ for $p\geq 2$.
Hence, we get that
$
\#\mathcal{A}'_{p^2}
\ll
1+\frac{B^2 \log B}{p^2}
$,
which along with the estimate
$\sum_{p>z}p^{-2}\ll
z^{-1}$
shows that
\[
\sum_{z<p\leq y}
\#\mathcal{A}'_{p^2}
\ll
\frac{B^2 \log B}{z}+y
.\]
This is the
required
estimate
since $Y$ is of order $B^2$.

Condition $\l(\text{A}\r)$:
It is shown in~\cite[\S 3.1,pg. 310]{marasingha}
that
$\varpi\!\l(p\r)<p$
for all $p>5$
and since $\{2,3,5\}\subseteq \overline{\mathfrak{P}}$
we get that
$\varpi\!\l(p\r)=0$
for $p=2,3$ and $5.$

Condition $\l(\text{B}\r)$:
Everything goes through as in
\cite[\S 3.2,pg.311]{marasingha}
where the dimension of the sieve is $\kappa=4$.

Condition $\l(\text{C}\r)$:
It appears that
the argument justifying
the analogous result
in
\cite[pg.314]{marasingha}
is not sufficiently detailed.
More specifically,
for values of $k$
that are small 
in comparison to
$Y$
the term
$\frac{Y^{1/2}2^Q}{k}$
in line $17$
has order of magnitude which is larger than the claimed bound
for $E\l(k\r)$
in line $21$.
Correcting this
inaccuracy
we get
\[E\l(k\r)
\ll
Q^{\nu}
\Big(2^Q
+\frac{\l(2^QY\r)^{1/2}}{k}
+
\frac{Y^{1/2}2^Q}{k}
\Big),
\]
for some absolute constant
$\nu$. Therefore
condition (C)
is satisfied
with
$\alpha=1/2-\eta$,
for any arbitrarily small and positive value of
$\eta$.

Condition $\l(\text{D}\r)$:
Follows directly
from
\eqref{coprim}.

Condition $\l(\text{E}\r)$:
Every element $a' \in \mathcal{A}'$
satisfies
$
a'
\ll
B^5,
$
where
the implied constant depends on
$F$ and $\mathcal{R}$.
Therefore there exists
$\mu>\frac{5}{\alpha}$,
which depends on
$F,\mathcal{R}$
and $\eta$,
such that
\[
|a'|
\leq Y^{\alpha \mu}
.
\]
Hence the conditions of
theorem
\cite[Th.3]{marasingha} are fulfilled.

Furthermore,
arguing similarly as in~\cite[\S6]{sarnak}, we obtain that one may take any value of $r$ larger than
$
m(\lambda):=
4\log \beta_4+
(5-\frac{1}{\beta_4}+\log\beta_4)\lambda-4\log\lambda-\lambda\log\lambda
$.
Note that
by~\cite[App.III]{DHR}
we have
$
\beta_4=9.0722\dots \ .
$
 We thus deduce
that
\[
\min_{0<\lambda<\beta_4}m(\lambda)=m(0.606519\dots)=15.4274522\dots
\]
and therefore
with the value
$r=16$
one has
\[
\#
\Big\{
(u,v) \in \Z^2 \cap \mathcal{R}B,
\Omega(\Phi'(u,v))\leq r
\Big\}
\gg
\frac{B^2}{\left(\log B\right)^4},
\]
as $B\to \infty$.
The
pairs of integers
$(u,v)\in \mathcal{R}B$
with
$\Omega(uv)
\leq
1$
are
such that $|uv|$ is $0,1$ or a prime,
thereby showing that
there are at most $\ll B$
such pairs.
We conclude that
\[
\#
\Big\{
(u,v) \in \Z^2 \cap \mathcal{R}B,
\Omega(\Phi'(u,v))\leq 16
\
\text{and}
\
\Omega(uv)
\geq 2
\Big\}
\gg
\frac{B^2}{\left(\log B\right)^4},
\]
as $B\to \infty$.
This observation
coupled
with
the equality
$
\Omega(\Phi(u,v))=
2 \
\Omega(\Phi'(u,v))-\Omega(uv)
$
completes the proof of Theorem
\ref{weighted}.
\end{proof}

We continue by using Lemmas~\ref{666} and~\ref{weighted}  to ensure
that
the $\Q$-rational fibre
$\pi^{-1}\!\l([s\!:\!t]\r)$
is birationally equivalent
to $\P_\Q^1$
through
almost prime values.

The case corresponding to $n=2$ of~\cite[Th.1.1]{kollar} implies that
$\overline{X}\l(\R\r)=X$.
For any point
$\boldsymbol{\xi}=\l(\xi_0,\ldots,\xi_3\r)\in X\l(\R\r)$
let us define
\begin{align*}
M_{\boldsymbol{\xi}}:&=
b\!\l(\xi_0,\xi_1\r)\!\xi_2+
e\!\l(\xi_0,\xi_1\r)\!\xi_3
\\
&=-
a\!\l(\xi_0,\xi_1\r)\!\xi_2^2
-
d\!\l(\xi_0,\xi_1\r)\!\xi_2\xi_3
-
f\!\l(\xi_0,\xi_1\r)\!\xi_3^2
\end{align*}
and note that
the points satisfying
$ \xi_0 \xi_1 M_{\boldsymbol{\xi}} \neq 0$
are also Zariski dense.
Hence
in order to prove Zariski
density of almost prime points
it suffices
to prove that
given
any such
$\boldsymbol{\xi}$
and
any $\epsilon>0$
there exists
$B \in \mathbb{Q}^\ast$
and
$\b{x} \in \Z^4$
with $\Omega\l(|x_0\ldots x_3|\r)\leq 34$
such that $[\b{x}] \in X$
and
\begin{equation}
\label{wa}
\max_{0\leq i \leq 3}\Big|
\frac{x_i}{B}
-\xi_i
\Big|
<\epsilon.
\end{equation}

Let $\delta \in \l(0,1\r)$
be a constant
to be specified in due course
and consider the following box of $\R^2$,
\[\mathcal{B}:=
[\xi_0-\delta,
\xi_0+\delta]
\!
\times
\!
[\xi_1-\delta,
\xi_1+\delta].
\]
The prime number theorem for arithmetic progressions
\cite[\S 5.6]{kowalski}
 supplies
a large positive number $P$
such that
the box
$P \mathcal{B}$
contains
a pair of, not necessarily positive, primes
$(s,t)$
lying in the progressions $\mod W$
specified by Lemma~\ref{gauss}.
A standard application of the mean value theorem
yields
\begin{equation}
\label{meanvalue}
\Big|
b\!\l(\frac{s}{P},\frac{t}{P}\r)\!\xi_2+
e\!\l(\frac{s}{P},\frac{t}{P}\r)\!\xi_3
-M_{\boldsymbol{\xi}}
\Big|
\ll_{X,\boldsymbol{\xi}}
\delta
\end{equation}
and
\begin{equation}
\label{meanvalue2}
\Big|
a\!\l(\frac{s}{P},\frac{t}{P}\r)\!\xi_2^2+
d\!\l(\frac{s}{P},\frac{t}{P}\r)\!\xi_2\xi_3+
f\!\l(\frac{s}{P},\frac{t}{P}\r)\!\xi_3^2+
M_{\boldsymbol{\xi}}
\Big|
\ll_{X,\boldsymbol{\xi}}
\delta,
\end{equation}
where the implied constant in both bounds is independent of $P$ and $\delta$.
Let
\[
\mathcal{R}:=
\Big[
\frac{\xi_2-\delta}{P \sqrt{M_{\boldsymbol{\xi}}}}
,
\frac{\xi_2+\delta}{P \sqrt{M_{\boldsymbol{\xi}}}}
\Big]
\!
\times
\!
\Big[
\frac{\xi_3-\delta}{P \sqrt{M_{\boldsymbol{\xi}}}}
,
\frac{\xi_3+\delta}{P \sqrt{M_{\boldsymbol{\xi}}}}
\Big].
\]
Applying
Lemma~\ref{666}
or
Lemma~\ref{weighted}
respectively
according to as if $\l(d^2-4af\r)\!\l(s,t\r)$
is an integer square or not,
provides
integers
$(u,v) \in B^{\frac{1}{2}}
\mathcal{R}
$
such that
$
uv
H^2_{s,t}\l(u,v\r)
G^2_{s,t}\l(u,v\r)
$
has respectively at most 
$8$ or
$30$ prime factors counted with multiplicity.
Define $\b{x} \in \Z^4$
by
\begin{align*}
x_0&=-s H_{s,t}(u,v), \\
x_1&=-t H_{s,t}(u,v), \\
x_2&=u \ G_{s,t}(u,v), \\
x_3&=v \ G_{s,t}(u,v).
\end{align*}
A further application of the mean value theorem
and a use of
\eqref{meanvalue}-
\eqref{meanvalue2}
shows that
\begin{align*}
\max_{0 \leq i \leq 3}
\Big|
\frac{x_i}{B}-\xi_i
\Big|
\ll_{X,\boldsymbol{\xi}}\delta,
\end{align*}
thereby proving
\eqref{wa}
for an
appropriate
$\delta:=\delta\l(\epsilon,X,\boldsymbol{\xi}\r).$
\begin{remark}
\label{rem:greentaozieg}
It is worth pointing out that 
a special corollary of the work in this section
is
that   whenever 
$a_i,e_i,f_i,(i=0,1)$
$b_j,d_j(j=0,1,2)$ are 
$12$
integers with 
\[
\underset{i=0,1}{\gcd}(a_i,d_i,f_i)=
\underset{j=0,1,2}{\gcd} (b_i,e_i)=1,
\gcd(6,a_0b_0)=1,
6 \mid \gcd(a_1,b_1,b_2),
6 \mid \underset{i=0,1}{\gcd}(d_i,e_i,f_i)\]
and the binary 
integer quadratic form
\[
(d_0s+d_1 t)^2-4(a_0s+a_1 t)(f_0s+f_1 t) 
\]
is the square of a polynomial in $\Q[s,t]$,
then the cubic surface $X$
given by the vanishing of 
\begin{align*}
&(a_0 x_0+a_1 x_1)x_2^2   +    (d_0 x_0+d_1 x_1)x_2 x_3  +   (f_0 x_0+f_1 x_1)x_3^2 
\\ &
+(b_0 x_0^2+b_1 x_0 x_1+b_2 x_0^2) x_2 +(e_0 x_0^2+e_1 x_0 x_1+e_2 x_0^2) x_3 
\end{align*}
is smooth and has
saturation number 
\[
r(X)\leq 10.
\]
It is easily seen that~\eqref{ter}
constitutes a special case of the surfaces alluded to 
in the present remark.
This 
proves the claim regarding the saturation of the surface in~\eqref{ter}
that was
made in the introduction.
\end{remark}

\section{The proof of Theorem~\ref{threefold}}
The Fermat cubic threefold
is 
equipped with a conic bundle over $\P^2_\Q$
and 
we begin by describing its structure.
The change of variables
\[
y_0=x_0,
y_1=\frac{x_1+x_3}{2},
y_2=\frac{x_2+x_4}{2},
y_3=\frac{x_1-x_3}{2},
y_4=\frac{x_2-x_4}{2},
\]
transforms the hypersurface
\[
X:
x_0^3+
x_1^3+
x_2^3+
x_3^3+
x_4^3=0
\]
into
\begin{equation}
\label{espresso}
Y: y_0^3+2y_1\l(y_1^2+3y_3^2\r)+2y_2\l(y_2^2+3y_4^2\r)=0.
\end{equation}
Let us record
here that any $y \in Y\l(\R\r)$
has
$y_0,y_1,y_2$ not of the same sign.
Projecting away from the line
$y_0=y_1=y_2=0$
contained in $Y$
provides
a conic bundle
morphism
$\pi:Y \to \P_{\Q}^2$.
The transformation
$\b{y}=(rx,sx,tx,y,z)$
shows that the fibres
$Q_{r,s,t}=0$
are
given by
\[x^2\l(r^3+2s^3+2t^3\r)
+y^2\l(6s\r)
+z^2\l(6t\r)
=0.
\]

We wish to
construct an infinite family of conic fibres
each member of which is $\Q$-rational and with a
parametrisation for their rational points
which is easy to describe.
We focus on the conics given by
$Q_
{p_1,-p_2^2,p_3^2}
=0$,
where $p_i$ are primes.
Note
that they contain the obvious point
$(0,p_3,p_2)$,
thus providing the
parametrisation  
\begin{align*}
&x=12p_2 p_3^2 uv, \\
&y=6p_3^3u^2+p_3\l(p_1^3-2p_2^6+2p_3^6\r)v^2,\\
&z=-6p_2p_3^2u^2+p_2\l(p_1^3-2p_2^6+2p_3^6\r)v^2.
\end{align*}
Tracing the substitutions backwards
allows us to conclude that the points
$(x_0,\ldots x_4)$
given by
\[
\l(
12p_1p_2p_3^2f_0\!\l(u,v\r),
p_3 f_1\!\l(u,v\r),
p_3 f_2\!\l(u,v\r),
p_2 f_3\!\l(u,v\r),
p_2 f_4\!\l(u,v\r)
\r)
\]
lie in the threefold $X$,
where
the binary integer forms
$f_j \in \Z[u,v]$
are defined through
\begin{align*}
&f_0:=uv,\\
&f_1:=6p_3^2u^2-12p_2^3p_3uv+\l(p_1^3-2p_2^6+2p_3^6\r)v^2,\\
&f_2:=-6p_3^2u^2-12p_2^3p_3uv-\l(p_1^3-2p_2^6+2p_3^6\r)v^2,\\
&f_3:=-6p_3^2u^2+12p_3^4uv+\l(p_1^3-2p_2^6+2p_3^6\r)v^2,\\
&f_4:=6p_3^2u^2+12p_3^4uv-\l(p_1^3-2p_2^6+2p_3^6\r)v^2.\\
\end{align*}
Observe
that
all
$\b{x} \in \Z^5$
on $X$
satisfy
$7|\prod_{i=0}^4 x_i$.
We have
\[
\prod_{i=0}^4
x_i=84p_1p_2^3p_3^4\l(\frac{1}{7}\prod_{i=0}^4f_i\l(u,v\r)\r).
\]
Now suppose that we are given a
point $\l(\xi_0,\ldots,\xi_4\r) \in X\l(\mathbb{R}\r)$
which we shall approximate.
The linear change of variables
\[\zeta_0=\xi_0,
\zeta_1=\frac{\xi_1+\xi_3}{2},
\zeta_2=\frac{\xi_2+\xi_4}{2},
\zeta_3=\frac{\xi_1-\xi_3}{2},
\zeta_4=\frac{\xi_2-\xi_4}{2}
\]
shows that
$\boldsymbol{\zeta} \in Y\!\l(\R\r)$
thus allowing us to assume without loss of generality that
\[
\zeta_0>0,\zeta_1<0,\zeta_2>0.
\]
owing to an earlier observation.

Now let $\delta$ be a small positive constant satisfying $0<\delta<1$.
For a large constant $C>0$, the
Siegel--Walfisz theorem
allows us to pick
primes $p_1,p_2,p_3$
satisfying
\begin{alignat}{2}
&p_i\equiv 1 \mod{q}
\
\text{for all}
\
i=1,2,3
\
\text{for all}
\
q=2,3,5,49, \label{progression} \\
&(p_1,-p_2^2,p_3^2) \in C\!\l((\zeta_0-\delta,\zeta_0+\delta)
\times (\zeta_1-\delta,\zeta_1+\delta) \times (\zeta_2-\delta,\zeta_2+\delta)\r), \label{prog0}\\
&p_3 \nmid p_1^3-2p_2^6, \label{prog1}\\
&p_1^3-8p_2^6+2p_3^6 \neq 0, \label{prog2}\\
&p_3\l(p_1^3-2p_2^6+2p_3^6\r)
\neq 0, \label{prog3}
\end{alignat}
With this choice of $p_i$
we shall
apply the weighted sieve,
as in the proof of Lemma~\ref{weighted},
in order to deduce that the form
\[F(u,v):=\frac{1}{7}\prod_{i=0}^4f_i\l(u,v\r)\]
attains almost prime values infinitely often.
Firstly, for $i=1,\ldots,4$,
each
$f_i$
is irreducible,
\[\text{disc}(f_1)=
\text{disc}(f_2)=24p_3^2\l(-p_1^3
+8p_2^6-2p_3^6\r).
\] The condition~\eqref{progression}
yields
$-p_1^3
+8p_2^6-2p_3^6\equiv 2 \mod{3}$ and hence
$\text{disc}(f_1)\neq 0$. Furthermore the same condition gives
$3\|\text{disc}(f_1)$,
i.e. $f_1$ and $f_2$ are irreducible over $\Q$.
A similar argument shows that the same holds true for both $f_3$ and $f_4$.
We next deduce that for $i=1,\ldots,4$
since the assumptions
$p_i\equiv 1 \mod{6},i=1,2,3$
yield that
\[\gcd\l(6,p_1^3-2p_2^6+2p_3^6\r)=1\]
and it suffices to note that
$p_3\nmid
\l(p_1^3-2p_2^6\r).$
Condition (C) is then proved as follows.
Let us define
\[
D:=
\text{rad}
\Big(
\Big(\prod_{p<10}p \Big)
\Big(\prod_{i=1}^{4} a_ic_i \text{disc}(f_i) \Big)
\Big( \prod_{1\leq i,j \leq 4}\text{Res}(f_i,f_j)  \Big)
\Big),
\]
where $a_i$
and
$c_i$
are
the coefficient of $u^2$
and
$v^2$
in $f_i$ respectively.
Note that~\eqref{prog3} guarantees that $u$ or $v$ are coprime forms to each $f_i, \forall i\neq 0$.
We need to find some $\b{z}$
such that
\[\l(u,v\r)\equiv \b{z} \mod{D}\]
in order to have
\begin{equation}
\label{need}
\gcd\l(\frac{1}{7}\prod_{i=0}^4f_i(u,v),D\r)=1.
\end{equation}
For the primes $p=2,3,5,7$ we make the choices
for $\l(u,v\r) \mod{p}$ given by
\[\l(1,1\r),
\l(1,1\r),
\l(1,2\r),
\l(1,1\r)
\]
respectively.
For the primes $p\geq 11$ we have the estimate
\[
\sharp\{u,v \in \l(\Z/p\Z\r)^*\!,p|\prod_{i=0}^4f_i(u,v)\}
\leq 8(p-1),
\]
which implies that we can pick an admissible vector $\l(u,v\r)\mod{p}$.
Alluding to the Chinese remainder theorem
we deduce that there exists $\b{z}$ satisfying~\eqref{need} 
thus allowing any value $r$ satisfying
\[r>
(3+6\log \beta_6)+(10-\frac{4}{\beta_6}+\log\beta_6)\lambda-6\log\lambda-\lambda\log\lambda:=m\l(\lambda\r)
\]
to work.
Note that condition $Q_0$ is verified as previously, the key element of this verification being that each $f_i,i\neq 0$ is
irreducible, and that $f_i,f_j$ are coprime for $i\neq j$.
Indeed since $f_i,f_j$ are coprime as polynomials, there are non-zero integers $R_{i,j}$
such that if $p|f_i(u,v)$ and $p|f_j(u,v)$ then $p|R_{i,j}u$ and $p|R_{i,j}v$.
Hence
\[
\#\{u,v \sim B:p|(f_i(u,v),f_j(u,v)\}
\ll
\#\{u,v \sim B:p|(u,v)\}
\ll
\frac{B^2}{p^2}+1
.
\]
Using ~\cite[App.III]{DHR}
we
deduce
that
\[
\min_{0<\lambda<\beta_6}m(\lambda)=m(0.4978357377\dots)=29.1527037101
\dots,
\]
so we can
choose $r=30.$

In order to finalise the Zariski density argument we need to prove that
$
|x_i/B-\xi_i|<\epsilon
$,
for which it suffices to establish that
$
|y_i/B-\zeta_i|<\epsilon/2$. 
To this end suppose that $(\alpha,\beta)$
is the solution to the system
\[\alpha \beta=\frac{1}{12\sqrt{-\zeta_1}\zeta_3}
\ 
\text{ and   } \ 
6\zeta_3 \alpha^2
+\l(\zeta_0^3+2\zeta_1^3+2\zeta_3^3\r)
\beta^2
=\frac{\zeta_2}{\sqrt{\zeta_3}}.
\]
Then the choices
\[
u \in \l(\l(\alpha-\delta\r)C^{-\frac{3}{4}}B^{\frac{1}{2}},\l(\alpha+\delta\r)C^{-\frac{3}{4}}B^{\frac{1}{2}}\r)
\ \text{ and } \ 
v \in \l(\l(\beta-\delta\r)C^{-\frac{7}{4}}B^{\frac{1}{2}},\l(\beta+\delta\r)C^{-\frac{7}{4}}B^{\frac{1}{2}}\r)
\]
make the inequality
$|
y_i(u,v)/B-\zeta_i|
<\delta
\
M_i\!\l(\boldsymbol{\xi}\r)
$
available.
We deduce that the saturation number satisfies
\[r(F)\leq 42.\]
The final stage 
of the 
proof of 
Theorem~\ref{threefold}
is similar 
to that of 
Theorem~\ref{2 skew}
and is therefore omitted.

\end{document}